\documentclass[10pt, reqno]{amsart}
\numberwithin{equation}{section}
\usepackage{amssymb, color, cite}
\usepackage{hyperref}

\newcommand{\qtq}[1]{\quad\text{#1}\quad}

\newcommand{\R}{\mathbb{R}}
\newcommand{\C}{\mathbb{C}}

\newcommand{\eps}{\varepsilon}

\newcommand{\F}{\mathcal{F}}

\newtheorem{theorem}{Theorem}[section]

\newtheorem{proposition}[theorem]{Proposition}

\theoremstyle{definition}

\newtheorem{remark}[theorem]{Remark}

\theoremstyle{remark}

\begin{document}

\title[Stability estimate]{Stability estimates for the recovery of the nonlinearity from scattering data}

\begin{abstract} We prove stability estimates for the problem of recovering the nonlinearity from scattering data.  We focus our attention on nonlinear Schr\"odinger equations of the form
 \[
 (i\partial_t+\Delta)u = a(x)|u|^p u
 \]
 in three space dimensions, with $p\in[\tfrac43,4]$ and $a\in W^{1,\infty}$. 
\end{abstract}

\author{Gong Chen}
\address{Georgia Institute of Technology, Atlanta, GA, USA}
\email{gc@math.gatech.edu}

\author{Jason Murphy}
\address{Missouri University of Science \& Technology, Rolla, MO, USA}
\email{jason.murphy@mst.edu}

\maketitle

\section{Introduction}

We consider the problem of determining an unknown nonlinearity from the small-data scattering behavior of solutions in the setting of nonlinear Schr\"odinger equations of the form
\begin{equation}\label{nls}
(i\partial_t+\Delta) u = a(x)|u|^p u,\quad (t,x)\in\R\times\R^d.
\end{equation}
We focus on the three-dimensional \emph{intercritical} setting, i.e. $d=3$ and $p\in[\tfrac43,4]$.  In this setting, equation \eqref{nls} admits a small-data scattering theory in $H^1$ for any $a\in W^{1,\infty}$ (see Theorem~\ref{T:direct} below).  In particular, given sufficiently small $u_-\in H^1$, there exists a global-in-time solution $u$ to \eqref{nls} that scatters backward in time to $u_-$ and forward in time to some $u_+\in H^1$, that is,
\[
\lim_{t\to\pm\infty} \|u(t)-e^{it\Delta}u_\pm\|_{H^1} = 0. 
\]
One can therefore define the \emph{scattering map} $S_a:u_-\mapsto u_+$ on some ball in $H^1$.  

As it turns out, the scattering map encodes all of the information about the nonlinearity in \eqref{nls}, in the sense that the map $a\mapsto S_a$ is injective (see Theorem~\ref{T:inverse}).  In fact, knowledge of $S_a$ suffices to reconstruct the inhomogeneity $a$ pointwise.

In this work, we consider the closely-related problem of stability.  That is, if two scattering maps $S_a$ and $S_b$ are close in some sense, must the corresponding inhomogeneities $a$ and $b$ necessarily be close?  Our main result (Theorem~\ref{T:stable} below) provides an estimate of this type.  It is essentially a quantitative version of Theorem~\ref{T:inverse}.

We first state the small data scattering result for \eqref{nls}.  The proof utilizes Strichartz estimates and a standard contraction mapping argument.  For completeness, we provide the proof in Section~\ref{S:first_results} below. 

\begin{theorem}[Small data scattering]\label{T:direct} Let $a\in W^{1,\infty}(\R^3)$ and $p\in[\tfrac43,4]$.  There exists $\eta>0$ sufficiently small so that for any $u_-\in H^1$ satisfying $\|u_-\|_{H^1}<\eta$, there exists a unique global solution $u$ to \eqref{nls} satisfying
\begin{align*}
\lim_{t\to\pm\infty}\|u(t)-e^{it\Delta}u_\pm\|_{H^1} = 0,
\end{align*}
where $u_+$ satisfies the formula
\begin{equation}\label{u_+}
u_+ = u_- - i\int_\R e^{-it\Delta}a|u|^p u(t)\,dt. 
\end{equation}
\end{theorem}

Using Theorem~\ref{T:direct}, we define the \emph{scattering map} $S_a:B\to H^1$ via $S_a(u_-)=u_+$, where $B$ is a ball in $H^1$ and $u_\pm$ are as in the statement of the theorem.  This map uniquely determines the function $a$ (see e.g. \cite{Strauss, Murphy}):

\begin{theorem}\label{T:inverse} Let $p\in[\tfrac43,4]$ and let $a,b\in W^{1,\infty}(\R^3)$.  Let $S_a,S_b$ denote the corresponding scattering maps for \eqref{nls} with nonlinearities $a|u|^p u$ and $b|u|^p u$, respectively. If $S_a=S_b$ on their common domain, then $a=b$. 
\end{theorem}

Our main result is essentially a quantitative version of Theorem~\ref{T:inverse}.  To measure the difference between two scattering maps, we use the Lipschitz constant at $0$.  In particular, we define
\[
\|S_a-S_b\| := \sup\biggl\{\frac{\|S_a(\varphi)-S_b(\varphi)\|_{H^1}}{\|\varphi\|_{H^1}}:\varphi\in B\backslash\{0\}\biggr\}
\]
where $B\subset H^1$ is the common domain of $S_a$ and $S_b$. 

\begin{theorem}[Stability estimate]\label{T:stable} Let $p\in[\tfrac43,4]$.  Let $a,b\in W^{1,\infty}$, and let $S_a,S_b$ denote the corresponding scattering maps for \eqref{nls} with nonlinearities $a|u|^p u$ and $b|u|^p u$, respectively. Then
\begin{align*}
\|a-b\|_{L^\infty} & \lesssim \{\|a\|_{W^{1,\infty}}+\|b\|_{W^{1,\infty}}\}^{\frac{8}{9}}\|S_a-S_b\|^{\frac1{9}} \\
&\quad +\{\|a\|_{W^{1,\infty}}+\|b\|_{W^{1,\infty}}\}^{\frac{10}{9}}\|S_a-S_b\|^{\frac8{9}}.
\end{align*}
\end{theorem}

\begin{remark} If we assume \emph{a priori} bounds of the form 
\[
\|a\|_{W^{1,\infty}},\|b\|_{W^{1,\infty}}\lesssim M\qtq{and}\|S_a-S_b\|\ll 1,
\] 
then the estimate in Theorem~\ref{T:stable} reduces to the following H\"older estimate: 
\begin{equation}\label{holder}
\|a-b\|_{L^\infty}\lesssim_M \|S_a-S_b\|^{\frac{1}{9}}.
\end{equation}
The precise powers appearing in these estimates do not have any special meaning.  Indeed, they arise from some \emph{ad hoc} choices made in the argument in order to treat the range $p\in[\tfrac43,4]$ uniformly. By refining the arguments, one could improve the estimate \eqref{holder} to
\[
\|a-b\|_{L^\infty}\lesssim_{M,\eps} \|S_a-S_b\|^{\frac{3p-2}{9p-2}-\eps},
\]
but even in this case there seems to be no special meaning to this exponent.  \end{remark}

The problem of recovering an unknown nonlinearity from scattering data (or other data) is a well-studied problem.  For results of this type in the setting of nonlinear dispersive equations (particularly nonlinear Schr\"odinger equations), we refer the reader to 
\cite{SBUW, SBS2, CarlesGallagher, ChenMurphy, EW, HMG, KMV, MorStr, Murphy, PauStr, Sasaki, Sasaki2, SasakiWatanabe, Watanabe0, Watanabe, Weder0, Weder1, Weder6, Weder3, Weder4, Weder5}.  To the best of our knowledge, the problem of stability has not yet been investigated in this particular setting; however, we refer the reader to \cite{LLPT} to some stability estimates related to recovering an unknown coefficient in a semilinear wave equation from the Dirichlet-to-Neumann map.  

Our main result, Theorem~\ref{T:stable}, provides a stability estimate in the intercritical setting for nonlinearities of the form $a(x)|u|^p u$ in three space dimensions.  The work \cite{KMV} proves an analogue of Theorem~\ref{T:inverse} for a more general class of nonlinearities in two dimensions; however, the results presented here do not suffice to establish a stability estimate in this more general setting.  In the case that modified scattering holds, the recent work \cite{ChenMurphy} also shows that the small-data modified scattering behavior also suffices to determine the inhomogeneity present in the nonlinearity.  A stability estimate in this setting would also require some new ideas compared to what is presented here.

The strategy of the proof of Theorem~\ref{T:stable} builds on the one used to prove Theorem~\ref{T:inverse} (see e.g. \cite{Strauss, Murphy}).  The starting point is the implicit formula for the scattering map appearing in \eqref{u_+}, which implies that
\[
\langle S_a(u_-)-u_-,u_-\rangle = -i\int_{\R\times\R^3}a(x) |u(t,x)|^p u(t,x)\,\overline{e^{it\Delta}u_-(x)}\,dx\,dt, 
\]
where $u$ is the solution to \eqref{nls} that scatters backward in time to $u_-$.  We then approximate the full solution $u(t)$ by $e^{it\Delta}u_-$ (the Born approximation), using the Duhamel formula for \eqref{nls} to express the difference (see \eqref{Duhamel}).  The difference contains the nonlinearity and hence is smaller than the main term, which is given by 
\begin{equation}\label{intro-main-term}
\int_{\R\times\R^3} a(x)|e^{it\Delta}u_-(x)|^{p+2}\,dx\,dt. 
\end{equation}

The next step is to specialize to Gaussian data of the form
\[
u_-(x) = \exp\{-\tfrac{|x-x_0|^2}{4\sigma^2}\},
\]
which is small in $H^1$ for $0<\sigma\ll1$.  We then rely on the fact that the free evolution of a Gaussian may be computed explicitly (and is still Gaussian), a fact that has already been exploited in the related works \cite{KMV, ChenMurphy, Murphy}.  Using the scaling symmetry for the linear Schr\"odinger equation, we can therefore express the main term \eqref{intro-main-term} in the form $F_\sigma\ast a(x_0)$, where $c^{-1}\sigma^{-5}F_\sigma$ forms a family of approximate identities as $\sigma\to 0$ for suitable $c>0$.  Using the explicit form of $F_\sigma$, we can estimate the difference 
\[
|c^{-1}\sigma^{-5}F_\sigma\ast a(x_0) - a(x_0)|
\]
quantitatively in terms of $\sigma$ (see Proposition~\ref{P:approx-id}).  Carrying out the same estimates with $S_b$ ultimately leads to a bound of the form
\[
\|a-b\|_{L^\infty} \lesssim \sigma^{-2}\|S_a-S_b\| + \mathcal{O}\{\sigma^{\frac14}+\sigma^2\},
\]
where $\sigma^{\frac14}$ arises from the approximate identity estimate and $\sigma^2$ arises from the Born approximation.  Optimizing with respect to $\sigma$ leads to the estimate appearing in Theorem~\ref{T:stable}. 

Theorem~\ref{T:stable} concerns the comparison of nonlinearities of the form $a|u|^p u$ and $b|u|^p u$; in particular, the power of each nonlinearity is \emph{a priori} assumed to be equal.  In fact, the result Theorem~\ref{T:inverse} (the determination of the nonlinearity from the scattering map) can be extended to allow nonlinearities of the form $a|u|^p u$ without assuming that $p$ is already known.  In particular, one can show that if nonlinearities $a(x)|u|^p u$ and $b(x)|u|^\ell u$ have the same scattering map, then $p=\ell$ and $a\equiv b$ (see e.g. \cite{Murphy, Watanabe}).  Thus it is also natural to ask whether one can bound $|p-\ell|$ in terms of the difference between the scattering maps.  

In this paper we also take the preliminary step of estimating $|p-\ell|$ in terms of the difference between the scattering maps corresponding to the pure power-type nonlinearities $|u|^pu$ and $|u|^\ell u$.

\begin{theorem}\label{T:pq} Suppose $p,\ell\in[\tfrac43,4]$.  Let $S_p$ and $S_\ell$ denote the scattering maps for \eqref{nls} corresponding to nonlinearities $|u|^p u$ and $|u|^\ell u$. Then 
\[
|p-\ell| \lesssim \|S_p-S_\ell\|^{\frac{1}{9}}.
\]
\end{theorem}

The proof of Theorem~\ref{T:pq} begins along similar lines to the proof of Theorem~\ref{T:stable}.  In the present setting, one needs to analyze the normalizing constant $\lambda(p)$ arising in the approximate identity argument mentioned above (see Proposition~\ref{P:approx-id}).  In particular, we derive an upper bound on $|\lambda(p)-\lambda(\ell)|$ in terms of $\|S_p-S_\ell\|$, and then establish a lower bound of the form
\[
|\lambda(p)-\lambda(\ell)|\gtrsim |p-\ell|. 
\]

Combining the arguments used to prove Theorem~\ref{T:stable} and Theorem~\ref{T:pq}, one can also obtain an estimate of the form
\[
\|\lambda(p)a - \lambda(\ell)b\|_{L^\infty} \lesssim \|S_1-S_2\|^{\frac{1}{9}},
\]
where $S_1,S_2$ are the scattering maps corresponding to \eqref{nls} with nonlinearities $a|u|^pu $ and $b|u|^\ell u$, respectively.  While this estimate is harder to interpret directly, it can still be used to prove that if $S_1=S_2$ then $p=\ell$ and $a\equiv b$ (recovering results of \cite{Watanabe, Murphy}).



The rest of this paper is organized as follows: In Section~\ref{S:prelim}, we collect some preliminary results.  We also prove the approximate identity result Proposition~\ref{P:approx-id}.  In Section~\ref{S:first_results}, we prove the small-data scattering result for \eqref{nls}. In Section~\ref{S:main}, we prove the main result, Theorem~\ref{T:stable}.  Finally, in Section~\ref{S:pq} we prove Theorem~\ref{T:pq}. 

\subsection*{Acknowledgements} J.M. was supported by NSF grant DMS-2137217. We are grateful to Gunther Uhlmann for suggesting that we consider stability estimates in this setting, as well as to John Singler for some helpful suggestions.

\section{Preliminaries}\label{S:prelim}
We write $A\lesssim B$ to denote $A\leq CB$ for some $C>0$, with dependence on parameters indicated by subscripts. We write $W^{1,\infty}$ for the Sobolev space with norm
\[
\|a\|_{W^{1,\infty}}=\|a\|_{L^\infty}+\|\nabla a\|_{L^\infty}. 
\]
For $1<r<\infty$ we write $H^{s,r}$ for the Sobolev space with norm
\[
\|u\|_{H^{s,r}} = \|\langle \nabla \rangle^s u\|_{L^r},
\]
where $\langle\nabla\rangle = \sqrt{1-\Delta}$.  We write $q'$ for the H\"older dual of an exponent $q$, i.e. the solution to $\tfrac{1}{q}+\tfrac{1}{q'}=1$. 

We write $e^{it\Delta}$ for the Schr\"odinger group $e^{it\Delta}=\F^{-1} e^{-it|\xi|^2}\F$, where $\F$ denotes the Fourier transform. 

We utilize the following Strichartz estimates \cite{GinibreVelo, KeelTao, Strichartz} in three space dimensions.

\begin{proposition}[Strichartz, \cite{GinibreVelo, KeelTao, Strichartz}] For any $2\leq q,\tilde q,r,\tilde r\leq\infty$ satisfying
\[
\tfrac{2}{q}+\tfrac{3}{r}=\tfrac{2}{\tilde q}+\tfrac{3}{\tilde r} = \tfrac{3}{2},
\]
we have
\begin{align*}
\|e^{it\Delta}\varphi\|_{L_t^q L_x^r(\R\times\R^3)}&\lesssim \|\varphi\|_{L^2}, \\
\biggl\| \int_{-\infty}^t e^{i(t-s)\Delta} F(s)\,ds\biggr\|_{L_t^q L_x^r(\R\times\R^3)}&\lesssim \|F\|_{L_t^{\tilde q'}L_x^{\tilde r'}(\R\times\R^3)}.
\end{align*}
\end{proposition}

The following approximate identity estimate plays a key role in both Theorem~\ref{T:stable} and Theorem~\ref{T:pq}.  It is based on the explicit computation of the solution to the linear Schr\"odinger equation with Gaussian data. We present the result in the setting of general dimensions and short-range powers. 

\begin{proposition}[Approximate identity estimate]\label{P:approx-id} Let $d\geq 1$ and $p>\tfrac{2}{d}$. 
 Given $x_0\in\R^d$ and $\sigma>0$, define 
\[
\varphi_{\sigma,x_0}(x) = \exp\{-\tfrac{|x-x_0|^2}{4\sigma^2}\}
\]
and
\begin{equation}\label{lambda}
\lambda(d,p) := \pi^{\frac{d}{2}+1}\bigl[\tfrac{4}{p+2}]^{\frac{d}{2}} \tfrac{\Gamma(\frac{dp}{4}-\frac12)}{\Gamma(\frac{dp}{4})}.
\end{equation}
Given $a\in W^{1,\infty}(\R^d)$, we have
\begin{align*}
\biggl| & \iint_{\R\times\R^d} |e^{it\Delta}\varphi_{\sigma,x_0}(x)|^{p+2} a(x) \,dx\,dt - \sigma^{d+2}\lambda(d,p)a(x_0)\biggr| \leq c_s \sigma^{d+2+s}\|a\|_{W^{1,\infty}}
\end{align*}
for any $0<s<1-\tfrac{2}{dp}$, where $c_s\to\infty$ as $s\to 1-\tfrac{2}{dp}$. 
\end{proposition}

\begin{proof} We have
\[
e^{it\Delta}\varphi_{\sigma,x_0}(x) = \bigl[\tfrac{\sigma^2}{\sigma^2+it}\bigr]^{\frac{d}{2}}\exp\bigl\{-\tfrac{|x-x_0|^2}{4(\sigma^2+it)}\bigr\}
\]
(see \cite{Visan}), so that
\begin{align*}
|e^{it\Delta}\varphi_{\sigma,x_0}(x)|^{p+2}&=\bigl[\tfrac{\sigma^4}{\sigma^4+t^2}\bigr]^{\frac{d(p+2)}{4}}\exp\{-\tfrac{\sigma^2|x-x_0|^2(p+2)}{4(\sigma^4+t^2)}\bigr\} \\
& = K(\tfrac{t}{\sigma^2},\tfrac{x-x_0}{\sigma}),
\end{align*}
where
\[
K(t,x):=\bigl[\tfrac{1}{1+t^2}\bigr]^{\frac{d(p+2)}{4}}\exp\bigl\{-\tfrac{|x|^2(p+2)}{4(1+t^2)}\bigr\}. 
\]

We now show that $\int K(t,x)\,dx\,dt = \lambda(d,p)$.  To this end, we first recall the Gaussian integral
\[
\int_\R \exp\{-cy^2\}\,dy = (\tfrac{\pi}{c})^{\frac12}.
\]
We next use the change of variables $u=(1+t^2)^{-1}$ to obtain
\begin{align*}
\int_\R(1+t^2)^{-c}\,dt & = 2\int_0^\infty (1+t^2)^{-c} \,dt\\
& =\int_0^1 u^{c-\frac32}(1-u)^{-\frac12}\,du \\
& = B(\tfrac12,c-\tfrac12) = \tfrac{\Gamma(\frac12)\Gamma(c-\frac12)}{\Gamma(c)} = \pi^{\frac12}\tfrac{\Gamma(c-\frac12)}{\Gamma(c)}
\end{align*}
for $c>\tfrac12$, where $B$ is the Euler Beta function.  Thus
\begin{align*}
\int_{\R\times\R^d} K(t,x)\,dx\,dt & = [\tfrac{4\pi}{p+2}]^{\frac{d}{2}}\int_\R (1+t^2)^{-\frac{dp}{4}}\,dt \\
& = \pi^{\frac{d}{2}+1}\bigl[\tfrac{4}{p+2}]^{\frac{d}{2}} \tfrac{\Gamma(\frac{dp}{4}-\frac12)}{\Gamma(\frac{dp}{4})}=\lambda(d,p),
\end{align*}
where we have used the fact that $p>\frac{2}{d}$.  

We also observe that for any $R>0$ and any $0<s<\tfrac{dp}{2}-1$, we may estimate
\begin{equation}\label{K-weighted}
\begin{aligned}
\int_\R\int_{|x|>R} K(t,x)\,dx\,dt & \lesssim R^{-s} \iint |x|^{s} K(t,x)\,dx\,dt \\
& \lesssim R^{-s}\int(1+t^2)^{-\frac{dp}{4}+\frac{s}{2}}\,dt \lesssim_s R^{-s}.
\end{aligned}
\end{equation}

By a change of variables, we have
\begin{equation}\label{int-K-sigma}
\iint_{\R\times\R^d} K(\tfrac{t}{\sigma^2},\tfrac{x}{\sigma})\,dx\,dt = \sigma^{d+2}\lambda(d,p).
\end{equation}
Thus we can write
\begin{align}
\biggl|\iint_{\R\times\R^d}& |e^{it\Delta}\varphi_{\sigma,x_0}(x)|^{p+2}a(x)\,dx\,dt - \sigma^{d+2}\lambda(d,p)a(x_0)\biggr| \nonumber \\
& = \biggl|\iint_{\R\times\R^d} K(\tfrac{t}{\sigma^2},\tfrac{x}{\sigma})[a(x_0-x)-a(x_0)]\,dx\,dt\biggr| \nonumber \\
& \leq \int_\R\int_{|x|\leq \delta} K(\tfrac{t}{\sigma^2},\tfrac{x}{\sigma})|a(x_0-x)-a(x_0)|\,dx\,dt \label{approx-id-error1} \\
& \quad + \int_\R\int_{|x|>\delta} K(\tfrac{t}{\sigma^2},\tfrac{x}{\sigma})|a(x_0-x)-a(x_0)|\,dx\,dt, \label{approx-id-error2} 
\end{align}
where $\delta>0$ will be determined below.

By the fundamental theorem of calculus and \eqref{int-K-sigma}, we first estimate
\begin{align*}
\eqref{approx-id-error1} \lesssim \delta \sigma^{d+2}\|\nabla a\|_{L^\infty}.
\end{align*}
Next, we use \eqref{K-weighted} to obtain
\begin{align*}
\eqref{approx-id-error2} & \lesssim \sigma^{d+2}\|a\|_{L^\infty}\int_\R \int_{|y|>\frac{\delta}{\sigma}}K(t,y)\,dy\,dt \lesssim_s [\tfrac{\sigma}{\delta}]^{s} \sigma^{d+2} \|a\|_{L^\infty}
\end{align*}
for $0<s<\tfrac{dp}{2}-1$. Choosing $\delta=\sigma^{\frac{s}{1+s}}$ leads to
\[
\biggl|\iint_{\R\times\R^d}|e^{it\Delta}\varphi_{\sigma,x_0}(x)|^{p+2}a(x)\,dx\,dt - \sigma^{d+2}\lambda(d,p)a(x_0)\biggr|\lesssim_s \sigma^{\frac{s}{1+s}}\sigma^{d+2}\|a\|_{W^{1,\infty}}
\]
for any $0<s<\tfrac{dp}{2}-1$, which yields the result. 
\end{proof}

\section{Small-data scattering}\label{S:first_results}

In this section we prove the following small-data scattering result.

\begin{theorem}\label{T:direct-general} Let $a\in W^{1,\infty}(\R^3)$ and $p\in[\tfrac43,4]$. Define
\begin{equation}\label{exponents}
(q,r) = (p+2,\tfrac{6(p+2)}{3(p+2)-4}) \qtq{and} s_c=\tfrac32-\tfrac{2}{p}.
\end{equation}
There exists $\eta>0$ sufficiently small so that for any $u_-\in H^1$ satisfying $\|u_-\|_{H^1}<\eta$, there exists a unique global solution $u$ to \eqref{nls} and $u_+\in H^1$ satisfying the following:
\begin{align*}
\|u\|_{L_t^q L_x^r(\R\times\R^3)} &\lesssim \|u_-\|_{L^2},\\
\|\nabla u \|_{L_t^q L_x^r(\R\times\R^3)} &\lesssim \|u_-\|_{H^1}, \\
\| |\nabla|^{s_c} u\|_{L_t^q L_x^r(\R\times\R^3)} &\lesssim \|u_-\|_{\dot H^{s_c}},
\end{align*}
and
\[
\lim_{t\to\pm\infty}\|u(t)-e^{it\Delta}u_\pm\|_{H^1} = 0.
\]
\end{theorem}

\begin{proof} We construct $u$ to satisfy the Duhamel formula
\begin{equation}\label{Duhamel}
u(t) = \Phi u(t):= e^{it\Delta} u_- - i\int_{-\infty}^t e^{i(t-s)\Delta}a(x)|u|^p u(s)\,ds,
\end{equation}
where $\|u_-\|_{H^1}\leq \eta\ll 1$.  It suffices to prove that $\Phi$ is a contraction on a suitable complete metric space.  To this end, we fix $u_-\in H^1$ and define $X$ to be the set of functions $u:\R\times\R^d\to\C$ satisfying the bounds
\[
\|u\|_{L_t^q L_x^r} \leq 4C\|u_-\|_{L^2},\quad \|\nabla u\|_{L_t^q L_x^r} \leq 4C\|u_-\|_{H^1},\quad \||\nabla|^{s_c} u\|_{L_t^q L_x^r} \leq 4C\|u_-\|_{\dot H^{s_c}},
\]
where $q,r$ are defined in \eqref{exponents}, all space-time norms are over $\R\times\R^3$, and $C$ encodes implicit constants in inequalities such as Strichartz and Sobolev embedding.  We equip $X$ with the metric
\[
d(u,v) = \|u-v\|_{L_t^q L_x^r}
\]
and we define 
\begin{equation}\label{exponents2}
r_c = \tfrac{3p(p+2)}{4}, \qtq{so that} \dot H^{s_c,r}\hookrightarrow L^{r_c}. 
\end{equation}

For $u\in X$, we use Strichartz and H\"older, to estimate
\begin{align*}
\|\Phi u\|_{L_t^q L_x^r} & \lesssim \|u_-\|_{L^2} + \|a|u|^p u\|_{L_t^{q'}L_x^{r'}} \\
& \lesssim \|u_-\|_{L^2}+\|a\|_{L^\infty}\|u\|_{L_t^q L_x^{r_c}}^p\|u\|_{L_t^q L_x^r} \\
& \lesssim \|u_-\|_{L^2} + \eta^p\|a\|_{L^\infty} \|u_-\|_{L^2}
\end{align*}
Similarly, using the product rule and Sobolev embedding as well, 
\begin{align*}
\|&\nabla\Phi u\|_{L_t^q L_x^r} \\ & \lesssim \|u_-\|_{\dot H^1} + \|a\|_{L^\infty} \|u\|_{L_t^q L_x^{r_c}}^p \|\nabla u\|_{L_t^q L_x^r} + \|\nabla a\|_{L^\infty}\|u\|_{L_t^q L_x^{r_c}}^{p-1}\|u\|_{L_t^q L_x^r}\|u\|_{L_t^q L_x^{r_c}} \\ 
& \lesssim \|u_-\|_{\dot H^1} + \eta^p \|a\|_{L^\infty}\|u_-\|_{\dot H^1} + \eta^p \|\nabla a\|_{L^\infty}\| |\nabla|^{s_c} u\|_{L_t^q L_x^r} \\
& \lesssim \|u_-\|_{\dot H^1} + \eta^p\|a\|_{W^{1,\infty}}\|u_-\|_{H^1}. 
\end{align*}
Finally, we have
\begin{align*}
\||\nabla|^{\frac12}\Phi u\|_{L_t^q L_x^r}  &\lesssim \|u_-\|_{\dot H^{s_c}} + \|a|u|^p u\|_{L_t^{q'} H_x^{1,r'}} \\
& \lesssim \|u_-\|_{\dot H^{s_c}} + \|a\|_{W^{1,\infty}} \|u\|_{L_t^q L_x^{r_c}}^{p-1} \|u\|_{L_t^q H_x^{1,r}}\|u\|_{L_t^q L_x^{r_c}} \\
& \lesssim \|u_-\|_{\dot H^{s_c}}+ \|a\|_{W^{1,\infty}} \|u\|_{L_t^q L_x^{r_c}}^{p-1}\|u\|_{L_t^q H_x^{1,r}}\| |\nabla|^{s_c} u\|_{L_t^q L_x^r} \\
& \lesssim \|u_-\|_{\dot H^{s_c}} + \eta^p\|a\|_{W^{1,\infty}}\|u_-\|_{\dot H^{s_c}}. 
\end{align*}
It follows that for $\eta$ sufficiently small, $\Phi:X\to X$. 

To see that $\Phi$ is a contraction, we use Strichartz and H\"older to estimate as follows: for $u,v\in X$, 
\begin{align*}
\|\Phi u - \Phi v\|_{L_t^q L_x^r} & \lesssim \|a[u-v]\|_{L_t^{q'}L_x^{r'}} \\
& \lesssim \|a\|_{L^\infty}[\|u\|_{L_t^q L_x^{r_c}}^p+\|v\|_{L_t^q L_x^{r_c}}^p]\|u-v\|_{L_t^q L_x^r} \\
& \lesssim \eta^p\|a\|_{L^\infty} \|u-v\|_{L_t^q L_x^r},
\end{align*}
which shows that $\Phi$ is a contraction for $\eta$ sufficiently small. 

It follows that $\Phi$ has a unique fixed point $u\in X$, which is our desired solution. 

It is not difficult to show that $u$ scatters backward in time to $u_-$, and hence it remains to show that $e^{-it\Delta}u(t)$ has a limit in $H^1$ as $t\to\infty$.  To this end, we fix $t>s>0$ and use the estimates above to obtain
\begin{align*}
\|e^{-it\Delta}u(t)-e^{-is\Delta}u(s)\|_{H^1} & \lesssim \|a |u|^p u\|_{L_t^{q'}L_x^{r'}((s,t)\times\R^3)} \\
& \lesssim \|a\|_{W^{1,\infty}}\|u\|_{L_t^q L_x^{r_c}((s,t)\times\R^3)}^2\|u\|_{L_t^q H_x^{1,r}((s,t)\times\R^3)} \\
& \to 0 \qtq{as}s,t\to\infty.
\end{align*}
Thus $\{e^{-it\Delta}u(t)\}$ is Cauchy in $H^1$ as $\to\infty$ and hence has some limit $u_+\in H^1$ as $t\to\infty$.  In fact, from the Duhamel formula \eqref{Duhamel} we can obtain the implicit formula
\begin{equation}\label{u_+2}
u_+=\lim_{t\to\infty}e^{-it\Delta}u(t) = u_--i\int_\R e^{-is\Delta}a|u|^p u(s)\,ds.
\end{equation}
for the final state $u_+$.  \end{proof}

\begin{remark} By introducing some additional space-time norms into the argument, one can upgrade the estimate
\[
\|\nabla u\|_{L_t^q L_x^r} \lesssim \|u_-\|_{H^1}\qtq{to}\|\nabla u\|_{L_t^q L_x^r}\lesssim \|u_-\|_{\dot H^1}.
\]
However, this refinement is not needed in what follows, so we have opted to keep the argument as simple as possible above. \end{remark}

\section{Proof of Theorem~\ref{T:stable}}\label{S:main}

In this section we prove Theorem~\ref{T:stable}.

\begin{proof}[Proof of Theorem~\ref{T:stable}] We let $p\in[\tfrac43,4]$ and $a,b\in W^{1,\infty}$.  Let $S_a,S_b$ denote the scattering maps for \eqref{nls} with nonlinearities $a(x)|u|^pu$ and $b(x)|u|^p u$, respectively.  Given $\sigma>0$ and $x_0\in\R^3$, we define
\begin{equation}\label{def:gaussian}
\varphi_{\sigma,x_0}(x) = \exp\{-\tfrac{|x-x_0|^2}{4\sigma^2}\}.
\end{equation}
As
\begin{equation}\label{Hsbound}
\|\varphi_{\sigma,x_0}\|_{\dot H^s(\R^3)} \lesssim \sigma^{\frac32-s} \qtq{for}s\in\R,
\end{equation}
we have that $\varphi_{\sigma,x_0}$ belongs to the common domain of $S_a$ and $S_b$ for all $\sigma$ sufficiently small.  Using \eqref{u_+2}, we write
\begin{align*}
S_a(\varphi_{\sigma,x_0})& =\varphi_{\sigma,x_0} - i\iint_\R e^{-it\Delta}\bigl\{a\,|e^{it\Delta}\varphi_{\sigma,x_0}|^p e^{it\Delta}\varphi_{\sigma,x_0}\bigr\}\,dt \\
& \quad -i\int_\R e^{-it\Delta}\bigl\{a\,\bigl[|u|^p u - |e^{it\Delta}\varphi_{\sigma,x_0}|^pe^{it\Delta}\varphi_{\sigma,x_0}\bigr]\bigr\}\,dt,
\end{align*}
where $u$ is the solution to \eqref{nls} that scatters to $\varphi_{\sigma,x_0}$ as $t\to-\infty$ (cf. Theorem~\ref{T:direct}).  Similarly,
\begin{align*}
S_b(\varphi_{\sigma,x_0})& =\varphi_{\sigma,x_0} - i\int_\R e^{-it\Delta}\bigl\{b\,|e^{it\Delta}\varphi_{\sigma,x_0}|^p e^{it\Delta}\varphi_{\sigma,x_0}\bigr\}\,dt \\
& \quad -i\int_\R e^{-it\Delta}\bigl\{b\,\bigl[|v|^p v - |e^{it\Delta}\varphi_{\sigma,x_0}|^p e^{it\Delta}\varphi_{\sigma,x_0}\bigr]\bigr\}\,dt,
\end{align*}
where $v$ is the solution to the  NLS (with nonlinearity $b|v|^p v$) that scatters to $\varphi_{\sigma,x_0}$ as $t\to-\infty$. Thus
\begin{align}
\langle& S_a(\varphi_{\sigma,x_0})-S_b(\varphi_{\sigma,x_0}),\varphi_{\sigma,x_0}\rangle\label{mainLHS} \\
& = -i\iint_{\R\times\R^3} [a(x)-b(x)]|e^{it\Delta}\varphi_{\sigma,x_0}|^{p+2}\,dx\,dt \label{mainRHS}\\
& -i\iint_{\R\times\R^3} a(x)[|u|^p u-|e^{it\Delta}\varphi_{\sigma,x_0}|^p e^{it\Delta}\varphi_{\sigma,x_0}] \overline{e^{it\Delta}\varphi_{\sigma,x_0}}\,dx\,dt \label{error1} \\
& -i\iint_{\R\times\R^3} b(x)[|v|^p v - |e^{it\Delta}\varphi_{\sigma,x_0}|^p e^{it\Delta}\varphi_{\sigma,x_0}] \overline{e^{it\Delta}\varphi_{\sigma,x_0}}\,dx\,dt \label{error2}. 
\end{align}

The terms \eqref{error1} and \eqref{error2} are estimated as in the proof of Theorem~\ref{T:direct} (see \eqref{exponents} and \eqref{exponents2} for the definitions of $q,r,r_c$). We use H\"older, Strichartz, the Duhamel formula \eqref{Duhamel}, Sobolev embedding, Theorem~\ref{T:direct}, and \eqref{Hsbound} to obtain
\begin{align*}
\| a[&|u|^p u-|e^{it\Delta}\varphi_{\sigma,x_0}|^p e^{it\Delta}\varphi_{\sigma,x_0}]e^{it\Delta}\varphi_{\sigma,x_0}\|_{L_{t,x}^1} \\
& \lesssim \|a\|_{L^\infty} \|e^{it\Delta}\varphi_{\sigma,x_0}\|_{L_t^q L_x^r} \| |u|^p + |e^{it\Delta}\varphi_{\sigma,x_0}|^p\|_{L_t^{\frac{q}{p}} L_x^{\frac{r_c}{p}}} \|u(t)-e^{it\Delta}\varphi_{\sigma,x_0}\|_{L_t^q L_x^r} \\
& \lesssim \|a\|_{L^\infty}\|\varphi_{\sigma,x_0}\|_{L^2}\{\|u\|_{L_t^q L_x^{r_c}}^p + \|e^{it\Delta}\varphi_{\sigma,x_0}\|_{L_t^q L_x^{r_c}}^p\}\biggl\| \int_{-\infty}^t e^{i(t-s)\Delta}a(x)|u|^p u\,ds \biggr\|_{L_t^q L_x^r} \\
& \lesssim \|a\|_{L^\infty}\|\varphi_{\sigma,x_0}\|_{L^2}\|\varphi_{\sigma,x_0}\|_{\dot H^{s_c}}^p \|a |u|^p u\|_{L_t^{q'} L_x^{r'}} \\
& \lesssim \|a\|_{L^\infty}^2 \|\varphi_{\sigma,x_0}\|_{L^2}\|\varphi_{\sigma,x_0}\|_{\dot H^{s_c}}^p \|u\|_{L_t^q L_x^{r_c}}^p \|u\|_{L_t^q L_x^r} \\
& \lesssim \|a\|_{L^\infty}^2 \|\varphi_{\sigma,x_0}\|_{L^2}^2\|\varphi_{\sigma,x_0}\|_{\dot H^{s_c}}^{2p} \\
& \lesssim \sigma^7\|a\|_{L^\infty}^2. 
\end{align*}
Similarly,
\[
\|b[|v|^p v - |e^{it\Delta}\varphi_{\sigma,x_0}|^p e^{it\Delta}\varphi_{\sigma,x_0}]\|_{L_{t,x}^1} \lesssim \sigma^7\|b\|_{L^\infty}^2 . 
\]

For \eqref{mainLHS}, we use Cauchy--Schwarz and \eqref{Hsbound} to obtain
\begin{align*}
|\langle S_a(\varphi_{\sigma,x_0})-S_b(\varphi_{\sigma,x_0}),\varphi_{\sigma,x_0}\rangle| & \leq \|S_a(\varphi)-S_b(\varphi)\|_{\dot H^1}\|\varphi_{\sigma,x_0}\|_{\dot H^{-1}}\\
& \lesssim \|S_a-S_b\|\,\|\varphi_{\sigma,x_0}\|_{H^1}\|\varphi_{\sigma,x_0}\|_{\dot H^{-1}} \\
&\lesssim \sigma^3\|S_a-S_b\|
\end{align*}

For \eqref{mainRHS}, we make use of Proposition~\ref{P:approx-id} with $d=3$ and $s=\frac{1}{4}$. Using the fact that $p\geq\tfrac43$, this proposition implies that
\[
\biggl|  \iint [a-b]|e^{it\Delta}\varphi_{\sigma,x_0}|^{p+2}\,dx\,dt - c\sigma^5[a(x_0)-b(x_0)] \biggr|
 \lesssim \sigma^{5+\frac14}[\|a\|_{W^{1,\infty}}+\|b\|_{W^{1,\infty}}].
\]
where $c=\lambda(3,p)$. Combining this with the estimates for \eqref{error1}--\eqref{error2}, we deduce
\begin{align*}
|a(x_0)-b(x_0)|&  \lesssim \sigma^{-2}\|S_a-S_b\| +\sigma^{\frac14}\{\|a\|_{W^{1,\infty}}+\|b\|_{W^{1,\infty}}\} \\
& \quad + \sigma^2\{\|a\|_{L^\infty}^2+\|b\|_{L^\infty}^2\}. 
\end{align*}
If we now choose
\[
\sigma=\eps\cdot\biggl[\frac{\|S_a-S_b\|}{\|a\|_{W^{1,\infty}}+\|b\|_{W^{1,\infty}}}\biggr]^{\frac{4}{9}}
\]
for sufficiently small $\eps>0$, then we obtain
\begin{align*}
|a(x_0)-b(x_0)| & \lesssim \{\|a\|_{W^{1,\infty}} +\|b\|_{W^{1,\infty}}\}^{\frac{8}{9}}\|S_a-S_b\|^{\frac1{9}} \\
&\quad +\{\|a\|_{W^{1,\infty}}+\|b\|_{W^{1,\infty}}\}^{\frac{10}{9}}\|S_a-S_b\|^{\frac8{9}}.
\end{align*}
Taking the supremum over $x_0\in\R^3$ now yields the result.\end{proof}

\section{Proof of Theorem~\ref{T:pq}}\label{S:pq}

\begin{proof}[Proof of Theorem~\ref{T:pq}] The proof begins similarly to the proof of Theorem~\ref{T:stable}.

Let $S_p$ and $S_\ell$ denote the scattering maps corresponding to \eqref{nls} with nonlinearities $|u|^p u$ and $|u|^\ell u$, respectively, and define $\varphi_{\sigma}$ as in \eqref{def:gaussian} with $x_0=0$. We let $u,v$ denote the solutions to \eqref{nls} with nonlinearities $|u|^p u$ and $|v|^\ell v$ that scatter backward in time to $\varphi_{\sigma}$.  Arguing as in the proof of Theorem~\ref{T:stable}, we can write
\begin{align}
\iint_{\R\times\R^3}&\bigl[|e^{it\Delta}\varphi_{\sigma}|^{p+2}-|e^{it\Delta}\varphi_{\sigma}|^{\ell+2}\bigr]\,dx\,dt\label{pq-main} \\
& = i\langle S_p(\varphi_{\sigma})-S_\ell(\varphi_{\sigma}),\varphi_{\sigma}\rangle\label{pq-operators}\\
& \quad + \iint_{\R\times\R^3}\bigl[|u|^p u - |e^{it\Delta}\varphi_{\sigma}|^pe^{it\Delta}\varphi_{\sigma}\bigr]\overline{e^{it\Delta}\varphi_{\sigma}}\,dx\,dt \label{pq-error1} \\
& \quad + \iint_{\R\times\R^3}\bigl[|v|^\ell v-|e^{it\Delta}\varphi_{\sigma}|^\ell e^{it\Delta}\varphi_{\sigma}\bigr]\overline{e^{it\Delta}\varphi_{\sigma}}\,dx\,dt. \label{pq-error2} 
\end{align}
The estimates of \eqref{error1}--\eqref{error2} in the proof of Theorem~\ref{T:stable} apply to \eqref{pq-error1}--\eqref{pq-error2}, so that
\[
|\eqref{pq-error1}|+|\eqref{pq-error2}| \lesssim \sigma^7. 
\]
Similarly, estimating as we did for \eqref{mainRHS}, we have
\[
|\eqref{pq-operators}| \lesssim \sigma^2\|S_p-S_{\ell}\| 
\]

For \eqref{pq-main}, we use Proposition~\ref{P:approx-id}  with $d=3$ and $s=\frac{1}{4}$, which shows that
\[
\bigl| \eqref{pq-main} - \sigma^5[\lambda(p)-\lambda(\ell)] | \lesssim \sigma^{5+\frac14},
\]
where we abbreviate $\lambda(3,p)$ and $\lambda(3,\ell)$ by $\lambda(p)$ and $\lambda(\ell)$, respectively. It follows that
\begin{align*}
|\lambda(p)-\lambda(\ell)| & \lesssim \sigma^{-2}\|S_p-S_\ell\| + \sigma^{\frac14} + \sigma^2\\
&\lesssim \sigma^{-2}\|S_p - S_\ell\| + \sigma^{\frac14}. 
\end{align*}
Optimizing in $\sigma$ implies that
\[
|\lambda(p)-\lambda(\ell)|\lesssim \|S_p-S_\ell\|^{\frac{1}{9}},
\]
and thus the proof reduces to proving that
\begin{equation}\label{LB}
|\lambda(p)-\lambda(\ell)|\gtrsim |p-\ell|. 
\end{equation}
In fact, recalling the definition of $\lambda$ in \eqref{lambda}, a direct calculation shows that
\[
\lambda'(p) = -c\tfrac{1}{(p+2)^{\frac32}}\tfrac{\Gamma(\frac{3p}{4}-\frac12)}{\Gamma(\frac{3p}{4})}\bigl\{\tfrac{3}{2(p+2)}+\tfrac{3}{4}\bigl[\psi(\tfrac{3p}{4})-\psi(\tfrac{3p}{4}-\tfrac12)\bigr]\bigr\},
\]
where $\psi$ is the digamma function, i.e. $\psi(z) = \tfrac{\Gamma'(z)}{\Gamma(z)}$. 

By Gautschi's inequality (see e.g. \cite[Theorem~A, p. 68]{Rademacher}), we have
\[
\tfrac{\Gamma(\frac{3p}{4}-\frac12)}{\Gamma(\frac{3p}{4})}>(\tfrac{3p}{4})^{-\frac12}.
\]
Using the fact that $\psi$ is increasing on $(0,\infty)$, it follows that
\[
|\lambda'(p)|\geq \tfrac{3c}{2}(p+2)^{-\frac52}(\tfrac{3p}{4})^{-\frac12} \gtrsim 1 \qtq{uniformly for}p\in[\tfrac43,4]. 
\]
This implies \eqref{LB} and completes the proof of Theorem~\ref{T:pq}.
\end{proof}


\begin{thebibliography}{100}


\bibitem{SBUW} A. S\'a Barreto, G. Uhlmann, and Y. Wang, \emph{Inverse scattering for critical semilinear wave equations}.  Pure Appl. Anal. \textbf{4} (2022), no. 2, 191--223.

\bibitem{SBS2} A. S\'a Barreto and P. Stefanov, \emph{Recovery of a cubic non-linearity in the wave equation in the weakly non-linear regime.} Comm. Math. Phys. \textbf{392} (2022), no. 1, 25--53.

\bibitem{CarlesGallagher} R. Carles and I. Gallagher, \emph{Analyticity of the scattering operator for semilinear dispersive equations.} Comm. Math. Phys. \textbf{286} (2009), no. 3, 1181--1209.

\bibitem{ChenMurphy} G. Chen and J. Murphy, \emph{Recovery of the nonlinearity from the modified scattering map}. Preprint {\tt arXiv:arXiv:2304.01455}. 

\bibitem{EW} V. Enss and R. Weder, \emph{The geometrical approach to multidimensional inverse scattering.} J. Math. Phys. \textbf{36} (1995), no. 8, 3902--3921.

\bibitem{GinibreVelo}  J. Ginibre and G. Velo, \emph{Smoothing properties and retarded estimates for some dispersive evolution equations.} Comm. Math. Phys. \textbf{144} (1992), 163--188. 

\bibitem{HMG} C. Hogan, J. Murphy, and D. Grow, \emph{Recovery of a cubic nonlinearity for the nonlinear Schr\"odinger equation.} J. Math. Anal. Appl. \textbf{522} (2023), no. 1, Article 127016. 

\bibitem{KeelTao} M. Keel and T. Tao, \emph{Endpoint Strichartz estimates.}  Amer. J. Math. \textbf{120} (1998), no. 5, 955--980.


\bibitem{KMV} R. Killip, J. Murphy, and M. Visan, \emph{The scattering map determines the nonlinearity.} Proc. Amer. Math. Soc. \textbf{151} (2023), no. 5, 2543--2557. 


\bibitem{LLPT} M. Lassas, T. Liimatainen, L. Potenciano-Machado, T. Tyni, \emph{Uniqueness, reconstruction and stability for an inverse problem of a semi-linear wave equation.} J. Differential Equations \textbf{337} (2022), 395--435. 

\bibitem{LeeYu} Z. Lee and X. Yu, \emph{A note on recovering the nonlinearity for generalized higher-order Schr\"odinger equations.} Preprint {\tt arXiv:2303.06312}.

\bibitem{MorStr} C. S. Morawetz and W. A. Strauss, \emph{On a nonlinear scattering operator.} Comm. Pure Appl. Math. \textbf{26} (1973), 47--54.

\bibitem{Murphy} J. Murphy, \emph{Recovery of a spatially-dependent coefficient from the NLS scattering map.} Preprint {\tt arXiv:2209.07680}.

\bibitem{PauStr} B. Pausader and W. A. Strauss, \emph{Analyticity of the nonlinear scattering operator.} Discrete Contin. Dyn. Syst. \textbf{25} (2009), no. 2, 617--626.

\bibitem{Rademacher} H. Rademacher, \emph{Topics in Analytic Number Theory.} Edited by E. Grosswald, J. Lehner and M. Newman.  Die Grundlehren der mathematischen Wissenschaften, Band \textbf{169}. \emph{Springer-Verlag, New York-Heidelberg,} 1973. ix+320 pp.

\bibitem{Sasaki2} H. Sasaki, \emph{The inverse scattering problem for Schr\"odinger and Klein-Gordon equations with a nonlocal nonlinearity,} Nonlinear Analysis, Theory, Methods \& Applications \textbf{66} (2007), 1770--1781. 

\bibitem{Sasaki} H. Sasaki, \emph{Inverse scattering for the nonlinear Schr\"odinger equation with the Yukawa potential.} Comm. Partial Differential Equations \textbf{33} (2008), no. 7-9, 1175--1197. 

\bibitem{SasakiWatanabe} H. Sasaki and M. Watanabe, \emph{Uniqueness on identification of cubic convolution nonlinearity.} J. Math. Anal. Appl. \textbf{309} (2005), no. 1, 294--306. 

\bibitem{Strauss} W. A. Strauss, \emph{Nonlinear scattering theory.} In Scattering Theory in Mathematical Physics, edited by J. A. Lavita and J. P. Marchand. D. Reidel, Dordrecht, Holland/Boston, 1974, pp. 53--178. 



\bibitem{Strichartz} R. Strichartz, \emph{Restrictions of Fourier transforms to quadratic surfaces and decay of solutions of wave equations.} Duke Math. J. \textbf{44} (1977), no. 3, 705--714.


\bibitem{Watanabe0} M. Watanabe, \emph{Inverse scattering for the nonlinear Schr\"odinger equation with cubic convolution nonlinearity.} Tokyo J. Math. \textbf{24} (2001), no. 1, 59--67.

\bibitem{Watanabe} M. Watanabe, \emph{Time-dependent method for non-linear Schr\"odinger equations in inverse scattering problems.} J. Math. Anal. Appl. \textbf{459} (2018), no. 2, 932--944.

\bibitem{Weder0} R. Weder, \emph{Inverse scattering for the nonlinear Schr\"odinger equation.} Comm. Partial Differential Equations \textbf{22} (1997), no. 11-12, 2089--2103. 
 
\bibitem{Weder1} R. Weder, \emph{Inverse scattering for the non-linear Schr\"odinger equation: reconstruction of the potential and the non-linearity.} Math. Methods Appl. Sci. \textbf{24} (2001), no. 4, 245--25 

\bibitem{Weder6} R. Weder, \emph{$L^p$-$L^{p'}$ estimates for the Schr\"odinger equation on the line and inverse scattering for the nonlinear Schr\"odinger equation with a potential.} J. Funct. Anal. \textbf{170} (2000), no. 1, 37--68.

\bibitem{Weder3} R. Weder, \emph{Inverse scattering for the nonlinear Schr\"odinger equation. II. Reconstruction of the potential and the nonlinearity in the multidimensional case.} Proc. Amer. Math. Soc. \textbf{129} (2001), no. 12, 3637--3645. 

\bibitem{Weder4} R. Weder, \emph{Inverse scattering for the non-linear Schr\"odinger equation: reconstruction of the potential and the non-linearity.} Math. Methods Appl. Sci. \textbf{24} (2001), no. 4, 245--254.

\bibitem{Weder5} R. Weder, \emph{Multidimensional inverse scattering for the nonlinear Klein-Gordon equation with a potential.} J. Differential Equations \textbf{184} (2002), no. 1, 62--77. 

\bibitem{Visan} M. Visan, \emph{Dispersive Equations}, in ``Dispersive Equations and Nonlinear Waves, Oberwolfach Seminars'' \textbf{45}, Birkhauser/Springer Basel AG, Basel, 2014.






\end{thebibliography}
\end{document}